\documentclass[a4paper,12pt,oneside]{amsart}

\usepackage{amssymb}  
\usepackage{latexsym} 
\usepackage{comment}
\usepackage{color}
\usepackage{colonequals}

 \usepackage{epsfig}
\usepackage{tikz}

\definecolor{darkgreen}{rgb}{0,0.5,0}
\usepackage[
        colorlinks, citecolor=darkgreen,
        pdfauthor={Ingrid Bauer, Roberto Pignatelli},
        pdftitle={Rigid, non infinitesimally rigid surfaces},
        linktocpage        
]{hyperref}

\usepackage[alphabetic,backrefs,lite]{amsrefs} 

\usepackage[all]{xy} 

\usepackage{enumerate} 

\newtheorem{theorem}{Theorem}[section]
\newtheorem{proposition}[theorem]{Proposition}

\newtheorem{corollary}[theorem]{Corollary}
\newtheorem{lemma}[theorem]{Lemma}
\newtheorem{definition}[theorem]{Definition}
\newtheorem{remark}[theorem]{Remark}

\newtheorem*{theorem-non}{Main Theorem}
\newtheorem*{question-non}{Question}

\theoremstyle{remark}
\newtheorem{rem}[theorem]{Remark}

\newcommand\sB{{\mathcal B}}

\newcommand{\hol}{\ensuremath{\mathcal{O}}}

\newcommand{\PP}{\ensuremath{\mathbb{P}}}

\newcommand{\FF}{\ensuremath{\mathbb{F}}}

\newcommand{\ra}{\ensuremath{\rightarrow}}

\newcommand{\Cl}{\operatorname{Cl}}
\newcommand{\Sym}{\operatorname{Sym}}

\DeclareMathOperator{\Pic}{Pic}

\DeclareMathOperator{\Ext}{Ext}

\DeclareMathOperator{\divi}{div}

\DeclareMathOperator{\Def}{Def}

\DeclareMathOperator{\PSL}{PSL}

\DeclareMathOperator{\Aut}{Aut}

\DeclareMathOperator{\diag}{diag}

\DeclareMathOperator{\rk}{rk}
\DeclareMathOperator{\GL}{GL}

\DeclareMathOperator{\cone}{cone}

\numberwithin{equation}{section}

\begin{document}
\title[Rigid Manifolds of Kodaira Dimension 1]{Fermat's Cubic, Klein's Quartic and Rigid Complex Manifolds of Kodaira Dimension One}

\author{Ingrid Bauer, Christian Gleissner}

\thanks{
The second author wants to thank F. Catanese, D. Frapporti, P. Graf and T. Peternell for their interest and helpful discussions.
}

\begin{abstract} 
For each  $n \geq 3$ we provide an $n$-dimensional rigid compact complex manifold of Kodaira dimension $1$. First we constructed a series of singular quotients of 
products of $(n-1)$ Fermat curves with the Klein quartic, which are rigid. Then using toric geometry a suitable resolution of singularities is constructed and the deformation theories of the singular model and of the resolutions are compared, showing the rigidity of the resolutions.
\end{abstract}

\maketitle

\section*{Introduction}
A compact complex manifold is  called {\em rigid} if it has no nontrivial deformations. 
In \cite{rigidity} several notions  of rigidity have been discussed, the relations among them have been studied and many questions and conjectures have been proposed. Among those there was the following:

\begin{question-non}
Do there exist rigid compact complex manifolds of dimension $n \geq 3$ and Kodaira dimension $1$?
\end{question-non}
The aim of this article is to give a positive answer to this question.

In fact, we construct for each $n \geq 3$ a projective manifold $\hat{X}_n$ of dimension $n$ and Kodaira dimension $1$, which is  infinitesimally rigid (which by Kuranishi theory implies that $\hat{X_n}$ is rigid, cf. Remark \ref{kuranishi}).

For this we start considering  the Klein quartic $Q$ and the Fermat cubic $F$. Both can be realized as {\em triangle curves} (i.e. Galois covers of $\PP^1$ branched on $\{0,1,\infty \}$) with group 
$$
G=\langle s,t  ~ \big\vert ~ s^3=1, ~ t^7=1,~ sts^{-1} =t^4 \rangle  \simeq \mathbb Z/7 \rtimes_{\varphi} \mathbb Z/3.
$$

For $n \geq 3$ we consider then $X_n := (F^{n-1} \times Q)/G$, where $G$ acts diagonally on the product.

It turns out that $X_n$ is a normal projective variety with isolated canonical cyclic quotient singularities, Kodaira dimension $1$ and 
$$
H^1(X_n, \Theta_{X_n}) = H^1(F^{n-1} \times Q, \Theta_{F^{n-1} \times Q})^G = 0.
$$

By Schlessinger's result \cite{schlessinger} these singularities are rigid (in dimensions $\geq 3$), hence by the local global $\Ext$ spectral sequence this implies that 
$$
H^1(X_n, \Theta_{X_n}) \simeq \Ext(\Omega^1_{X_n},\hol_{X_n}) =0.
$$
Since $\Ext(\Omega^1_{X_n},\hol_{X_n})$ is the tangent space of the base of the Kuranishi familiy $\Def(X_n)$, this shows that $X_n$ is an infinitesimally rigid (singular) variety.

Since we are looking for rigid {\em manifolds}, we construct a suitable resolution $\rho \colon \hat{X}_n \ra X_n$ of singularities and show that 
$H^1(X_n, \Theta_{X_n}) = H^1(\hat{X}_n, \Theta_{\hat{X}_n})$. Therefore the main result in our paper is
\begin{theorem-non}\label{mainintro}
Let $n \geq 3$ and let $X_n = (F^{n-1} \times Q) /G$. 
Then $X_n$ is infinitesimally rigid and there exists a  resolution of singularities $\rho \colon \hat{X}_n \to X_n$, such that
\begin{enumerate}
\item $H^1(\hat{X}_n, \Theta_{\hat{X}_n}) = 0$, i.e. $\hat{X}_n$ is infinitesimally rigid;
\item $\kappa(\hat{X}_n) = 1$.
\end{enumerate}
\end{theorem-non}

The paper is organized as follows: in the first section we recall some of the different notions of rigidity given in \cite{rigidity}, their mutual relations and some of the results established in loc.cit.. In the second paragraph we study the action  of $G$ on the curves $Q$ and $F$ and introduce the quotient varieties $X_n$. In the third paragraph we show that the $G$ action of the product $F^{n-1} \times Q$ is infinitesimally rigid, i.e. $$H^1(F^{n-1} \times Q, \Theta_{F^{n-1} \times Q})^G = 0.$$ 
Finally, in the last section, we construct a suitable resolution of singularities $\rho \colon \hat{X}_n \to X_n$, with methods from toric geometry, and show 
that $\rho$ satisfies the statements of the above theorem. Hereby we  
conclude the proof of our main theorem \ref{mainintro}. 

\section{Rigid compact complex manifolds}\label{section1}

The aim of this section is to recall the basic notions of rigidity and some of the results established in \cite{rigidity}.

We state the part of  \cite{rigidity}*{Definition 2.1}, which will be relevant for our purposes:
\begin{definition}\label{rigid} \

Let $X$ be a compact complex manifold of dimension $n$. 
\begin{enumerate}
\item A {\em deformation of $X$} is a  proper smooth holomorphic map of pairs $$f \colon (\mathfrak{X},X)  \rightarrow (\mathcal{B}, b_0)
$$ 
where $(\sB,b_0)$ is a connected (possibly not reduced) germ of a complex space.

\item $X$ is said to be  {\em  rigid}  if for each deformation of $X$,
\[
f \colon (\mathfrak{X},X)  \rightarrow (\sB, b_0)
\] 
there is an open neighbourhood $U \subset \sB$ of $b_0$ such that $X_t := f^{-1}(t) \simeq X$ for all $t \in U$.
\item  $X$ is said to be  {\em infinitesimally rigid} if 
\[
H^1(X, \Theta_X) = 0,
\]
where $\Theta_X$ is the sheaf of holomorphic vector fields on $X$.
\end{enumerate}
\end{definition}

\begin{rem}\label{kuranishi} \

1) If $X$ is infinitesimally rigid, then $X$ is also locally rigid. This follows by Kodaira-Spencer-Kuranishi theory, since $H^1(X, \Theta_X)$ is the Zariski tangent space of the germ of analytic space which is the base $\Def(X)$ of the Kuranishi semiuniversal deformation of $X$.
So, if  $H^1(X, \Theta_X) =0$, $\Def(X)$ is a reduced point and all deformations are induced by the trivial deformation. The other implication does not hold in general  as it was shown in \cite{notinfinitesimally}, compare also \cite{kodairamorrow}.

2) Observe that, as it is shown in \cite[Theorem 2.3]{rigidity}, a compact complex manifold is rigid if and only if the base of the Kuranishi family $\Def(X)$  has dimension $0$.

3) The only rigid curve is $\PP^1$.
\end{rem}

For $n=2$ the following was shown in \cite[Theorem 1.3]{rigidity}.

\begin{theorem}
Let $S$ be a smooth compact complex  surface, which is rigid. Then either
\begin{enumerate}
\item $S$ is a minimal surface of general type, or
\item $S$ is a Del Pezzo surface of degree $d \geq 5$, $\PP^2 $ or $\PP^1 \times \PP^1 = \FF_0$,
$\FF_1 = S_8$, or $S_7, S_6, S_5$; where $S_{9-r}$ is the blow-up of $\PP^2 $
in $r$ points which are in general linear position.
\item $S$ is an Inoue surface of type $S_M$ or $S_{N,p,q,r}^{(-)}$ (cf. \cite{inoue}).
\end{enumerate}
\end{theorem}

In particular, a rigid compact complex surface has Kodaira dimension $- \infty$ or $2$. That this is a phenomenon in low dimensions and that in higher dimensions rigid manifolds are much more frequent has already been observed in \cite[Theorem 1.4]{rigidity}. In fact, the following is shown:
\begin{theorem} \
\begin{enumerate}
\item
For all $n \geq 3$ and $2 \leq k \leq n$ there is a rigid n-dimensional compact complex manifold $X_{n,k}$ of Kodaira dimension $k$. 
\item
For all $n \geq 4$ there exists a rigid $n$-dimensional compact complex manifold $X_n$ of Kodaira dimension $0$.
\end{enumerate}
\end{theorem}
That there exist rigid threefolds of Kodaira dimension $0$ had already been shown by A. Beauville (cf. \cite{beauville}).

The existence of rigid  $n$-dimensional complex manifolds of Kodaira dimension $1$ was conjectured, but up to now no examples have been known. The aim of this paper is to give for all 
$n \geq 3$ such an example, i.e. our main result is the following

\begin{theorem}\label{main}
For each $n\geq 3$ there is a $n$-dimensional infinitesimally rigid compact complex manifold $\hat{X}_n$ of Kodaira dimension $1$.
\end{theorem}
 As an immediate consequence we get the following
 
 \begin{corollary}
 There are rigid, but not infinitesimally rigid, manifolds of dimension $n \geq 5$ and Kodaira dimension $3$.
 \end{corollary}

 This complements the result of \cite[Theorem 5.1]{notinfinitesimally}
\begin{theorem}\label{higherdim}
There are rigid, but not infinitesimally rigid, manifolds of dimension $n$ and Kodaira dimension $\kappa$ for all possible pairs $(n,\kappa)$ with $n \geq 5$ and $\kappa \neq 0,1,3$ and for 
$(n,\kappa)=(3,-\infty),(4,-\infty),(4,4)$.
\end{theorem}

The idea to construct the infiniesimally rigid examples with Kodaira dimension $1$ is (similarly as in \cite{rigidity}) to consider finite  quotients of smooth compact complex manifolds with respect to a rigid holomorphic group action. In the above quoted paper it was sufficient to consider free actions, so the quotient was smooth. If we drop the freeness assumption of the action, under mild assumptions it is still true that the quotient is infinitesimal rigid (in dimension at least three), but since we are interested in infinitesimally rigid manifolds, we have to compare the infinitesimal deformations of the quotient with those of a suitable resolution of singularities.

We are going to be more precise:

\begin{definition}
Let $Y$ be a compact complex manifold and $G$ be a finite group of automorphisms acting on $Y$. We say that the group action is {\em infinitesimally rigid} if and only if 
$H^1(Y,\Theta_Y)^G=0$. 
\end{definition}
\begin{remark}\label{tracemap}
There is a natural sheaf homomorphism  $ \pi_{\ast} \Theta_Y \to \Theta_X$ called the trace map.
If $G$ acts freely in codimension one, then the trace map  induces an isomorphism $ (\pi_{\ast} \Theta_Y)^G  \simeq \Theta_X$, in particular
$H^i(X,\Theta_X) \simeq H^i(Y,\Theta_Y)^G$. 
\end{remark}

\begin{proposition}\label{5term}
Let $Y$ be a projective manifold and $G$ be a finite holomorphic group action which is free in codimension one. Let $\rho \colon \hat{X} \to X$ be a resolution of the quotient 
$X=Y/G$ such that 
\begin{enumerate}
\item $\rho_{\ast} \Theta_{\hat{X}}  =  \Theta_{X}$, 
\item $R^1 \rho_{\ast} \Theta_{\hat{X}}=0$,
\end{enumerate}
then $H^1(\hat{X},\Theta_{\hat{X}}) \simeq H^1(Y,\Theta_Y)^G$.  

In particular,  $\hat{X}$ is infinitesimally rigid if and only if the $G$-action on $Y$ is infinitesimally rigid. 
\end{proposition}

\begin{proof}
The five term exact sequence of Leray's spectral sequence 
$$E^{p,q}_2 :=H^p(X,R^q \rho_{\ast} \Theta_{\hat{X}}) \implies H^{p+q}(\hat{X}, \Theta_{\hat{X}})$$ reads
$$
0 \to E^{1,0}_2 \to H^1(\hat{X}, \Theta_{\hat{X}}) \to E^{0,1}_2 \to E^{2,0}_2 \to H^2(\hat{X}, \Theta_{\hat{X}}). 
$$
Under the assumptions made,  it yields an isomorphism 
$H^1(X, \Theta_{\hat{X}}) \simeq H^1(X, \Theta_{X}) $ and the claim follows using Remark \ref{tracemap}. 
\end{proof}

\newpage

\section{The quotient varieties $X_n$}

 {\bf The Klein quartic and the Fermat cubic}

In \cite{Klein} Felix Klein studied a remarkable symmetric smooth plane quartic
\[
Q=\lbrace x_0^3x_1 + x_1^3x_2 + x_2^3x_0=0 \rbrace \subset \mathbb P_{\mathbb C}^2.
\]
This curve of genus $3$ was the first example realizing the Hurwitz bound 
$$|\Aut(C)| \leq 84\big(g-1\big)$$
 for the number of automorphisms of a compact Riemann surface $C$ of genus $g$.
Its automorphism group is $\PSL(2,\mathbb F_7)$ of order $168$, which acts on $Q$ by a faithful three dimensional matrix representation, 
see Klein's original exposition, or \cite{Elk99} for a modern treatment. 
In the following, we consider the subgroup $G$ of $\Aut(Q)$ generated by the projective transformations 
\[
S=
\begin{pmatrix}
0 & 1 & 0 \\
0 & 0 & 1 \\
1 & 0 & 0 
\end{pmatrix} \qquad \makebox{and} \qquad 
T=
\begin{pmatrix}
\zeta^4 & 0 & 0 \\
0 & \zeta^2 & 0 \\
0 & 0 & \zeta 
\end{pmatrix}, \qquad \zeta:=\exp\bigg(\frac{2 \pi \sqrt{-1}}{7}\bigg)
\]
of order three and seven, respectively. 
As an abstract group,  $G$ is the unique non-abelian group of order $21$: 
\[
G=\langle s,t  ~ \big\vert ~ s^3=1, ~ t^7=1,~ sts^{-1} =t^4 \rangle  \simeq \mathbb Z/7 \rtimes_{\varphi} \mathbb Z/3. 
\]
Some of its basic properties are summarized in the following lemma. 

\begin{lemma}\label{strgrp}
The group $G$ has five conjugacy classes: the trivial one, two classes of elements of order three 
\[
\Cl(s)=\lbrace s,ts, \ldots, t^6s \rbrace, \qquad \Cl(s^2)=\lbrace s^2,ts^2, \ldots, t^6s^2\rbrace
\]
and two of elements of order seven
\[
\Cl(t)=\lbrace t, t^2, t^4 \rbrace, \qquad 
\Cl(t^3)=\lbrace t^3, t^5, t^6 \rbrace.
\] 
The Sylow $3$-subgroups  are the seven cyclic groups 
$\langle t^is \rangle = \lbrace 1,t^is,t^{5i}s^2 \rbrace$ and  the unique Sylow $7$-subgroup is the cyclic group $ \langle t\rangle$.  
\end{lemma}
\noindent

Another highly symmetric algebraic  curve is the Fermat cubic $$\lbrace x_0^3 + x_1^3 +x_2^3 =0\rbrace \subset \mathbb P^2.$$
This curve of genus 
one is biholomorphic to the complex torus 
\[
F:=\mathbb C/\Lambda, \qquad \makebox{where} \qquad \Lambda :=\mathbb Z + \mathbb Z \epsilon \qquad \makebox{with} \qquad   \epsilon :=\exp\bigg(\frac{2\pi \sqrt{-1}}{3}\bigg). 
\]
A $G$-action on $F$ is defined by the affine transformations 
\[
f_s(z)=\epsilon z, \qquad \makebox{and} \qquad f_t(z)=z+ \frac{1+ 3\epsilon}{7}.
\]
\begin{rem}
Note that the  translations $(f_t)^d$,  correspond to the following $7$-torsion points:
\[
d \cdot \frac{1+3\epsilon}{7},  \qquad \makebox{where} \qquad 1 \leq d \leq 6.
\]
\end{rem}
\noindent 

\noindent 
Next we calculate the points on $Q$ resp. $F$ with non-trivial stabilizer.

\begin{proposition} \label{stabs_orbs}
For both curves $Q$ and $F$, the points with non-trivial stabilizer form three $G$-orbits. 
The table below gives a representative of each orbit, its stabilizer, the action of 
 the generator of the stabilizer  in local holomorphic  coordinates and the length  of the orbit: 

\medskip
{

\begin{center}
\setlength{\tabcolsep}{1pt}
\begin{tabular}{lclc|}
\begin{tabular}{|c |}
\hline
point    \\
\hline
stabilizer    \\
\hline
~ local action  ~  \\
\hline
length    \\
\hline
\end{tabular}
& ~
\begin{tabular}{|c | c | c |}
\hline
 $ ~ (1:0:0)  ~ $ & $ ~ (\epsilon^2:1:\epsilon)  ~ $  & $ ~ (\epsilon:1:\epsilon^2) ~ $   \\
\hline
 $\langle t \rangle $ & $\langle s \rangle $  & $\langle s \rangle $  \\
\hline
  $x \mapsto \zeta^4 x$ & $x \mapsto \epsilon x$  & $x \mapsto \epsilon^2 x$  \\
\hline
  $3$ & $7$  & $7$  \\
\hline
\end{tabular}
& ~
\begin{tabular}{|c | c | c |}
\hline
 $   0  $ & $ ~ \frac{1+2\epsilon}{3}  ~ $  & $ ~ \frac{2+\epsilon}{3}  ~ $   \\
\hline
$ \langle s \rangle  $ & $\langle s \rangle $  & $\langle s \rangle $  \\
\hline
 $~ x \mapsto \epsilon x ~ $ & $~ x \mapsto \epsilon x ~$  & $~ x \mapsto \epsilon x ~$  \\
\hline
 $7$ & $7$  & $7$  \\
\hline
\end{tabular}
\end{tabular}
\end{center}

} 
\end{proposition}

\begin{proof}
First we consider the Klein quartic. 
The fixed points of $T$ are $(1:0:0)$, $(0:1:0)$ and  
$(0:0:1)$. These points are permuted by $S$ and stabilized by the powers of $T$, thus they form a single orbit.  
The orbit-stabilizer-correspondence  implies that the stabilizer of  $p_0:=(1:0:0)$ has order seven and is therefore  equal to $ \langle t \rangle$. 
To understand the local action around $p_0$, we consider the open affine
$\lbrace x_0=1\rbrace $, where the curve is given by $x+x^3y+y^3=0$ and $T$ acts via $(x,y) \mapsto (\zeta^5x,\zeta^4 y)$. By the implicit function theorem $y$ is a local parameter and 
we see that  
multiplication with $\zeta^4$ is the local action. 
The matrix $S$ has three 1-dimensional eigenspaces corresponding to the points $q_0:=(1:1:1)$, $q_1:=(\epsilon^2:1:\epsilon)$ and  
$q_2:=(\epsilon:1:\epsilon^2)$.  Only $q_1$ and $q_2$ belong to the quartic. The orbits of $q_1$ and $q_2$ are  the seven translates by the powers of $T$: 
\[
\lbrace (\zeta^{4d}\epsilon^2: \zeta^{2d}: \zeta^{d}\epsilon) ~ \big\vert ~ 0 \leq d \leq 6 \rbrace  \qquad \makebox{and} \qquad 
\lbrace (\zeta^{4d}\epsilon: \zeta^{2d}: \zeta^{d}\epsilon^2) ~ \big\vert ~ 0 \leq d \leq 6 \rbrace. 
\]
Hence, the stabilizers of $q_1$ and $q_2$ have order three, so they are equal to $ \langle s \rangle$. 
Note that $q_1$ and $q_2$ are contained  in the open affine 
$\lbrace x_1=1\rbrace $, where the curve is given by 
$x^3+y+y^3x=0$. The automorphism $S$ restricts to the open set 
  $\lbrace x \neq 0 \rbrace \cap \lbrace y \neq 0 \rbrace$,
 where we can  write it  as $(x,y) \mapsto (1/y,x/y)$. 
Recall that the local action of  $S$ around these points is the same as the action on the tangent spaces  
$T_{(\epsilon^2,\epsilon)}Q$ and $T_{(\epsilon,\epsilon^2)}Q$,  induced by multiplication with the Jacobian matrices: 
  \[
J_S(\epsilon^2,\epsilon)=
 \begin{pmatrix}
0 & -\epsilon  \\
\epsilon^2 & -1
\end{pmatrix} \qquad \makebox{and} \qquad 
  J_S(\epsilon,\epsilon^2)=
 \begin{pmatrix}
0 & -\epsilon^2  \\
\epsilon & -1
\end{pmatrix}.
\]  
Since both tangent spaces are equal to the  line $x+y=0$, we conclude that 
$J_S(\epsilon^2,\epsilon)$,  acts via multiplication by $\epsilon$ and $J_S(\epsilon,\epsilon^2)$ 
by multiplication with $\epsilon^2$. 
Let $p \in Q$ be a point with non-trivial stabilizer. As the stabilizer of $p$ is cyclic, it must be $\langle t \rangle$,  or one of the seven 
$3$-Sylow subgroups. In the first case $p$ is equal to $(1:0:0)$, $(0:1:0)$ or $(0:0:1)$ and in the second case the point $p$  is in the orbit of $q_1$ or $q_2$, because  the $3$-Sylow subgroups are conjugated. 

\noindent
Next we analyze the Fermat elliptic curve. Clearly the translations $f_{t^d}$ for $1 \leq d \leq 6$ cannot have fixed points and
the condition for $z$ to be a fixed point of $f_{s}$ is  $z(\epsilon-1) \in \Lambda$. 
We write $z= a + b\epsilon$ with $a,b \in \mathbb R$ and compute 
\[
z(\epsilon-1)=(a+b\epsilon)(\epsilon-1)= - (a+b) + (a-2b)\epsilon . 
\]
Hence $z(\epsilon-1)$ is in the lattice  if and only if $a+b$ and $a-2b$ are integers. 
Necessarily this implies that  $3a$ and $3b$ are also integers, so $z$ is one of the nine 3-torsion points. Out of these only $0$,  $\frac{1+ 2\epsilon}{3}$ and 
$\frac{2+\epsilon}{3}$
fulfill the conditions $a+b \in \mathbb Z$ and $a-2b \in \mathbb Z$. 
The local action of $s$ around any of these three points is multiplication  by $\epsilon$.  
Their orbits are the translates by $f_{t^d}$, where  $0 \leq d \leq 6$. 
We conclude that stabilizers of  $0$,  $\frac{1+ 2\epsilon}{3}$ and 
$\frac{2+\epsilon}{3}$ have order three, so they are equal to $\langle s \rangle$. 
As above, it follows that, apart from the 21 points in the orbits, there are no other  points with non-trivial stabilizer.  
\end{proof}

\noindent 
Using Hurwitz's formula we immediately get that $\pi_Q\colon Q \ra Q/G$ and $\pi_F\colon F \ra F/G$ are {\em triangle curves}. In fact, we have:

\begin{corollary}
The quotients $Q/G$ and $F/G$ are isomorphic to $\mathbb P^1$ and the $G$-covers
$\pi_Q\colon Q \to \mathbb P^1$ and $\pi_F\colon F \to \mathbb P^1$ are branched in three points  $0,1,\infty \in \mathbb P^1$. 
\end{corollary}

{\bf The singular quotients $X_n$} \ 

For each $n \geq 3$ we consider the diagonal action of $G$ on the product 
$F^{n-1} \times Q$ given by:
\[
(t^as^b)(z_1,\ldots,z_{n-1},y):=\big(f_{t^as^b}(z_1), \ldots, f_{t^as^b}(z_{n-1}),T^aS^by \big), \qquad  f_{t^as^b}:=f_{t}^a \circ f_{s}^b,
\]

and the quotient 
\[
X_n:=\big(F^{n-1} \times Q\big)/G.
\]
\begin{rem}
Observe that  $X_n$ is a normal, $\mathbb Q$-factorial projective variety. 
\end{rem}

Using Proposition \ref{stabs_orbs}, we can compute the singular locus of $X_n$ and its Kodaira dimension.
\begin{proposition}\label{singsandkod}
The variety $X_n$ has $3^{n-1}$ cyclic quotient singularities of type 
$\frac{1}{3}(1, \ldots,1)$ and $3^{n-1}$ singularities of type $\frac{1}{3}(1, \ldots,1,2)$. In particular $X_n$ has canonical singularities and 
Kodaira dimension $\kappa(X_n)= 1$. 
\end{proposition}

\begin{proof}
Let $[(z_1,\ldots,z_{n-1},y)] \in X_n$ be a singular point, then the stabilizer of the representative $(z_1,\ldots,z_{n-1},y)$ is one of the 
$3$-Sylow groups. Since they are all conjugated, we can assume that the stabilizer is $\langle s \rangle$. The representative is then 
unique, each $z_i$ must be one of the points 
$0$, $\frac{1+2\epsilon}{3}$ or $\frac{2+\epsilon}{3}$ and 
$y$ is either $(\epsilon^2:1:\epsilon)$ or $(\epsilon:1:\epsilon^2)$, by to Proposition \ref{stabs_orbs}. 
Therefore,  the local action of  $s$ is either  multiplication with the diagonal matrix  $\diag(\epsilon, \ldots, \epsilon)$ or 
$\diag(\epsilon, \ldots, \epsilon, \epsilon^2)$, depending on $y$. Hence, there are $3^{n-1}$ points of type  $\frac{1}{3}(1, \ldots,1)$ and $3^{n-1}$ of type   $\frac{1}{3}(1, \ldots,1,2)$. These singularities are canonical by the criterion of Reid-Shepherd-Barron-Tai see \cite[ p. 376 Theorem]{R87}. 
Since the quotient map $\pi \colon F^{n-1} \times Q \to  X_n$ is quasi-etale, we have 
$\kappa(X_n)= \kappa(F^{n-1} \times Q)=\kappa(Q)=1$, according to  \cite[Section 3]{catQED}. 
\end{proof}

\section{Rigidity of the $G$-action}

In this section we show that the $G$ action on $Y_n:=F^{n-1} \times Q$ is infinitesimally rigid i.e.  
$H^1(Y_n,\Theta_{Y_n})^G=0$. 
Our strategy is to determine the character $\chi_{\psi}$ of the natural $G$-representation 
\[
\psi \colon G \to \GL(H^1\big(Y_n,\Theta_{Y_n})^{\vee}\big)
\]
and show that $\chi_{\psi}$  does not contain the trivial character $\chi_{triv}$.

\begin{rem}
The representation theory of $G$ is easy to understand: according to Lemma \ref{strgrp}  there are $5$ conjugacy classes and consequently also $5$ irreducible representations. 
Apart from the trivial representations we obtain  two one-dimensional representations from the quotient  $G/\langle t \rangle \simeq \mathbb Z/3$ by inflation. 
As usual they are identified  with their characters: 
\[
\chi_{\epsilon}(t^as^b)= \epsilon^{b} \qquad \makebox{and}  \qquad \chi_{\epsilon^2}(t^as^b)= \epsilon^{2b}.
\]
Since the degrees $d_i$ of the remaining two representations satisfy $d_1^2+d_2^2=18$, we conclude that $d_1=d_2=3$. The matrices $S$ and $T$  from the previous section define a faithful $3$ dimensional representation $\eta$, which must be irreducible since $G$ is non-abelian. The second three-dimensional representation is the complex conjugate $\overline{\eta}$, explicitly:
\[
s \mapsto S=
\begin{pmatrix}
0 & 1 & 0 \\
0 & 0 & 1 \\
1 & 0 & 0 
\end{pmatrix} \qquad \makebox{and} \qquad 
t \mapsto \overline{T}=
\begin{pmatrix}
\zeta^3 & 0 & 0 \\
0 & \zeta^5 & 0 \\
0 & 0 & \zeta^6 
\end{pmatrix}.
\]
\end{rem}
\noindent 
These representations occur naturally in our geometric picture: the pullback induces representations on the global sections of the tensor powers 
of  the sheaf of holomorphic $1$-forms on the Klein quartic
\[
\psi_{\omega_Q^{\otimes k}} \colon G \mapsto \GL\big(H^0(Q,\omega_Q^{\otimes k})\big), \qquad t^as^b \mapsto [\omega \mapsto (S^{-b}T^{-a})^{\ast} \omega]
\]
and similarly on the Fermat cubic.  Explicitly, we have: 

\begin{lemma}\label{repsforms}
The characters of the representations on the $1$-forms and the quadratic differentials are 
\[
\chi_{\omega_Q}=\chi_{\overline{\eta}}, \qquad \chi_{\omega_Q^{\otimes 2}}=\chi_{\eta} + \chi_{\overline{\eta}}, \qquad 
\chi_{\omega_F}=\chi_{\epsilon^2} \qquad \makebox{and} \qquad \chi_{\omega_F^{\otimes 2}}=\chi_{\epsilon}.
\]
\end{lemma}

\begin{proof}
The Klein quartic $Q \subset \mathbb P^2$ is canonically embedded i.e. we can regard $x_0$, $x_1$ and $x_2$ as a basis of $H^0(Q,\omega_Q)$. 
With respect to this basis, 
the representation  $\psi_{\omega_Q}$ is then given by $\overline{\eta}$. By Max Noether's classical result \cite[Noether's Theorem page 253]{GriffH},  there is  a surjection 
$\Sym^2 H^0(Q,\omega_Q) \to H^0(Q,\omega_Q^{\otimes 2})$, which is in our case an isomorphism since both spaces are $6$-dimensional. Hence 
\[
\chi_{\omega_Q^{\otimes 2}}(g)= \frac{\chi_{\omega_Q}(g)^2+ \chi_{\omega_Q}(g^2)}{2}= \chi_{\eta}(g) + \chi_{\overline{\eta}}(g)
 \qquad \makebox{for all} \qquad g \in G. 
\]
For the Fermat cubic the situation is the following: the space $H^0(F,\omega_F)$ is spanned by the $1$-form $dz$. The computations 
$ \big(f_{s^{-1}}\big)^{\ast}dz = \epsilon^{-1} dz = \epsilon^{2} dz$  and $\big(f_{t^{-1}}\big)^{\ast}dz = dz$ 
imply $\chi_{\omega_F}=\chi_{\epsilon^2}$, but also $ \chi_{\omega_F^{\otimes 2}}=\chi_{\epsilon^2}^2=\chi_{\epsilon}$, 
since   $dz^{\otimes 2}$ is a basis of $H^0(F,\omega_F^{\otimes 2})$.
\end{proof}
\noindent 
The lemma above provides enough information to understand $\psi$: 

\begin{proposition}\label{rigidG}
The character of  $\psi$ is given by 
\[
\chi_{\psi}= (n-1)^2 \chi_{\epsilon} + n\chi_{\overline{\eta}} + \chi_{\eta},
\]
in particular $H^1(Y_n,\Theta_{Y_n})^G=0$.
\end{proposition}

\begin{proof}
In  the proof we write $F_j$ for the Fermat curve 
$F$ at position $j$ in the product $Y_n$ and  denote the projection to the $j$-th factor by $p_j$. 
The tangent bundle $\Theta_{Y_n}$ can be written as 
\[
\Theta_{Y_n}=\bigg(\bigoplus_{j=1}^{n-1} p_j^{\ast}\Theta_{F_j} \bigg)  \oplus  p_n^{\ast}\Theta_Q,  
\]
which implies 
\[
H^{1}(\Theta_{Y_n}) \simeq \bigg(\bigoplus_{j=1}^{n-1} H^1(p_j^{\ast}\Theta_{F_j})\bigg) \oplus H^1(p_n^{\ast}\Theta_Q). 
\]
The sum  can be decomposed further using  K\"unneth's formula:  
\[
H^1(p_j^{\ast}\Theta_{F_j}) \simeq H^{1}(\Theta_{F_j}) \oplus H^{0}(\Theta_{F_j}) \otimes 
\bigg(\bigoplus_{i \neq j}^{n-1} H^1(\mathcal O_{F_i}) \oplus H^1( \mathcal O_Q)  \bigg),   
\]
and
$H^1(p_n^{\ast}\Theta_{Q}) \simeq H^1(\Theta_Q)$.
Dualising yields:
$$H^1(p_j^{\ast}\Theta_{F_j})^{\vee} \simeq  H^0(\omega_{F_j}^{\otimes 2}) \oplus H^1(\omega_{F_j}^{\otimes 2}) 
\otimes 
\bigg(\bigoplus_{i \neq j}^{n-1} H^0(\omega_{F_i}) \oplus H^0( \omega_Q)  \bigg), $$
and $H^1(p_n^{\ast}\Theta_{Q})^{\vee} \simeq H^0(\omega_Q^{\otimes 2})$.
We point out that the $G$ action on the wedge product  $dz \wedge d\overline{z}$ is trivial, showing  that 
$$H^1(\omega_{F_j}^{\otimes 2})\simeq H^{1,1}(F_j,\omega_{F_j}) \simeq \langle (dz \wedge d\overline{z})\otimes dz \rangle$$
and $H^0(\omega_{F_j})$ are equivalent representations.
By Lemma \ref{repsforms}  the character of $\psi$ is 
$$\chi_{\psi}=(n-1) \big[\chi_{\epsilon} + (n-2)\chi_{\epsilon^2}^2 +\chi_{\epsilon^2}\chi_{\overline{\eta}}\big] + \chi_{\eta} + \chi_{\overline{\eta}}
= (n-1)^2 \chi_{\epsilon} + n\chi_{\overline{\eta}} + \chi_{\eta},$$
thanks to  the identity $\chi_{\epsilon^2} \chi_{\overline{\eta}}= \chi_{\overline{\eta}}$. 
 \end{proof}

\section{Toric resolutions and direct images of the tangent sheaf}

\noindent 
In this section we prove the following proposition:

\begin{proposition}\label{resproposition}
For each $n \geq 3$ there exists a resolution $\rho \colon \hat{X}_n \to X_n$ of singularities of the quotient variety $X_n= (F^{n-1} \times Q) /G$ 
with the following properties: 
\begin{enumerate}
\item $\rho_{\ast} \Theta_{\hat{X}_n} \simeq \Theta_{X_n}$, and  
\item $R^1\rho_{\ast} \Theta_{\hat{X}_n} =0$. 
\end{enumerate}
\end{proposition}

\noindent

The proposition leads immediately to a proof of theorem \ref{main}.
More precisely, we have:

\begin{theorem}
Let $n \geq 3$ and let  $\rho \colon \hat{X}_n \to X_n$ be the resolution of singularities (constructed in the above proposition) of $X_n = (F^{n-1} \times Q) /G$.
Then it holds:
\begin{enumerate}
\item $H^1(\hat{X}_n, \Theta_{\hat{X}_n}) = 0$, i.e., $\hat{X}_n$ is infinitesimally rigid;
\item $\kappa(\hat{X}_n) = 1$.
\end{enumerate}
\end{theorem}

\begin{proof}
1) By Proposition \ref{rigidG}  and \ref{5term}, it holds 
$0=H^1(Y_n,\Theta_{Y_n})^G = H^1(\hat{X}_n, \Theta_{\hat{X}_n})$.
2) Proposition \ref{singsandkod} tells us that $X_n$ has canonical singularities and $\kappa(X_n)=1$, thus 
$\kappa(\hat{X}_n) = 1$.
\end{proof}

Moreover we can prove the following:
\begin{corollary}
There are rigid, but not infinitesimally rigid, manifolds of dimension $n \geq 5$ and Kodaira dimension $3$.
\end{corollary}

\begin{proof}
In  \cite{notinfinitesimally} the authors construct for each even number $d \geq 8$, not divisible by $3$, a rigid regular smooth algebraic surface $S_d$ of general type 
with $H^1(S_d,\Theta_{S_d}) \simeq \mathbb C^6$. The product $S_d \times \hat{X}_n$ is a projective manifold of Kodaira dimension $3$. 
By K\"unneth's formula 
$$H^1(S_d \times \hat{X}_n,\Theta_{S_d \times \hat{X}_n})= H^1(S_d,\Theta_{S_d}) \oplus  H^1(\hat{X}_n,\Theta_{\hat{X}_n})  \simeq \mathbb C^6,$$
because $S_d$ is regular and of general type. Thus $S_d \times \hat{X}_n$ is not infinitesimally rigid.
However, the product $S_d \times \hat{X}_n$ is rigid, because the factors are rigid and the base of the  Kuranishi family is a product 
$$ \Def(S_d \times \hat{X}_n)= \Def(S_d) \times \Def(\hat{X}_n)$$ according to  \cite[Lemma 5.2]{notinfinitesimally}.
\end{proof}

\noindent
\begin{rem}\label{resrem}
 \
\begin{enumerate}
\item
To construct a resolution of $X_n=(F^{n-1} \times Q) /G$ with the properties \emph{(1)} and \emph{(2)} of Proposition \ref{resproposition} is a local problem, because 
the singularities $\frac{1}{3}(1,\ldots,1)$  and $\frac{1}{3}(1,\ldots,1,2)$
of $X_n$ are isolated. Locally, the germs of these singularities are represented by affine toric varieties. This allows us to use tools from 
toric geometry to construct such a resolution.  The basic references in toric geometry are \cite{F93} and \cite{CLS11}.
\item
For any resolution $\rho \colon \hat{X} \to X$  of a normal variety $X$,  the direct image 
$\rho_{\ast} \Theta_{\hat{X}}$  is a subsheaf  of the reflexive sheaf $\Theta_{X}$. This inclusion is an equality if and only if $\rho_{\ast} \Theta_{\hat{X}}$ is reflexive. 
Observe that even in very simple situations the inclusion can be strict:  e.g. take the  blowup of the origin of $\mathbb C^2$. For $n=2$ compare \cite[Proposition (1.2)]{burnswahl}.
\item
Similarly, the vanishing of $R^1\rho_{\ast} \Theta_{\hat{X}}$ for a resolution  $\rho \colon \hat{X} \to X$  of a normal variety $X$ is not automatic: 
take the resolution of an $A_1$ surface singularity by a $-2$ curve, then  $R^1\rho_{\ast} \Theta_{\hat{X}}$ is a skyscraper sheaf at the singular point with value 
$H^1\big(\mathbb P^1,\mathcal O(-2)) \simeq \mathbb C$.
 More generally, for ADE surface singularities $R^1\rho_{\ast} \Theta_{\hat{X}}$  is never zero, compare
 \cite{burnswahl}, \cite{pinkham}, \cite{schlessinger}.
 \end{enumerate}
\end{rem}
 
 \noindent
 {\bf The toric blowup of  $\frac{1}{3}(1,\ldots,1)$}
 
 \noindent
 The singularity $\frac{1}{3}(1,\ldots,1)$ is the affine toric varitey 
 $U$ given by the lattice 
 \[
 N=\mathbb Z e_1 + \ldots + \mathbb Z e_{n-1} + \frac{\mathbb Z}{3}(1, \ldots ,1),  
 \]
 and the cone 
 $\sigma=\cone(e_1,\ldots ,e_n)$.
The star subdivision of $\sigma$ along the ray generated by $v:=\frac{1}{3}(1, \ldots ,1)$ yields a fan  
$\Sigma$ with the following $n$-dimensional cones 
\[
\sigma_i := \cone(e_1, \ldots, \widehat{e_i}, \ldots e_n,v),  \ \  1 \leq i \leq n. 
\]
For $n=3$ the picture is:

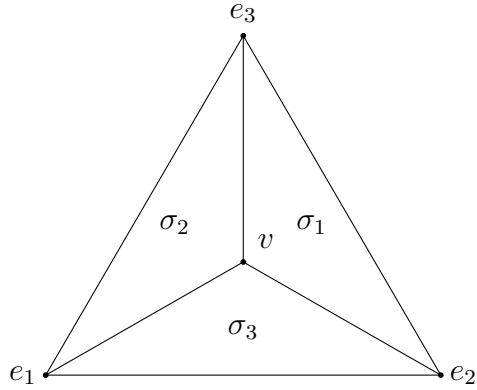
\begin{center}
\begin{tikzpicture}

\draw[fill] (0,3) circle [radius=0.03];
\draw[fill] (-2.598,-1.5) circle [radius=0.03];
\draw[fill] (2.598,-1.5)  circle [radius=0.03];
\draw[fill] (0,0) circle [radius=0.03];

\node at (0,3.3) {$e_3$};
\node at (-2.898,-1.5) {$e_1$};
\node at (2.898,-1.5) {$e_2$};
\node at (0.3,0.27) {$v$};

\draw (-2.598,-1.5) -- (2.598,-1.5); 
\draw (2.598,-1.5) -- (0,3); 
\draw (0,3) -- (-2.598,-1.5); 

\draw (-2.598,-1.5) -- (0,0); 
\draw (2.598,-1.5) -- (0,0); 
\draw (0,3) -- (0,0); 

\node at (0,-0.9) {$\sigma_3$};
\node at (0.9,0.5) {$\sigma_1$};
\node at (-0.9,0.5) {$\sigma_2$};
\end{tikzpicture}
\end{center}

Since the cones $\sigma_i$ are smooth, the subdivision induces a resolution $\rho \colon U_{\Sigma} \to U$, where $U_{\Sigma}$ is 
the toric variety of the fan $\Sigma$.  The resolution admits a single exceptional prime divisor $E$: 
it is the divisor corresponding to the ray $\mathbb R_{\geq 0} v$.  In the sequel, we denote the divisors 
corresponding to the  rays $\mathbb R_{\geq 0}e_i$ by $D_i$.
The resolution is called the \emph{toric blowup} of  $\frac{1}{3}(1,\ldots,1)$.

\noindent  
{\bf The Danilov resolution of $\frac{1}{3}(1,\ldots,1,2)$}
 
\noindent 
 The singularity $\frac{1}{3}(1,\ldots,1,2)$ is the affine toric variety $U$ given by the lattice 
 $$
 N=\mathbb Z e_1 + \ldots  + \mathbb Z e_{n-1} + \frac{\mathbb Z}{3}(1, \ldots ,1,2)
 $$
and the cone $\sigma=\cone(e_1,\ldots ,e_n)$. 

 The star subdivision along the ray generated by 
$v=\frac{1}{3}(1,\ldots,1,2)$ yields a fan with maximal cones 
\[
\sigma_i=\cone(e_1, \ldots,  \widehat{e_i}, \ldots, e_n,v), \ \ 1 \leq i \leq n. 
\]
All of these cones are smooth, with the exception of $\sigma_n$. Indeed, for $i \neq n$,  the vectors  
$\lbrace e_1, \ldots, \widehat{e_i}, \ldots, e_n,v\rbrace$ form a $\mathbb Z$-basis of $N$, but $e_n \notin \langle e_1, \ldots, e_{n-1},v \rangle_{\mathbb Z}$.
Therefore, we need a further subdivision of  $\sigma_n$, this time along the ray generated by 
$v'=\frac{1}{3}(2,\ldots,2,1) \in \sigma_n$. 
The maximal cones are:
\[
\tau_i =\cone(e_1, \ldots,  \widehat{e_i},  \ldots, e_{n-1},v,v'), \ \ 1 \leq i \leq n-1
\]
and $\tau_n=\cone(e_1,\ldots ,e_{n-1},v')$. The picture below illustrates the subdivision  in dimension three:

\begin{center}
\begin{tikzpicture}

\draw[fill] (0,3) circle [radius=0.03];
\draw[fill] (-2.598,-1.5) circle [radius=0.03];
\draw[fill] (2.598,-1.5)  circle [radius=0.03];
\draw[fill] (0,0) circle [radius=0.03];

\draw[fill] (0,1.5) circle [radius=0.03];

\node at (0,3.3) {$e_3$};
\node at (-2.898,-1.5) {$e_1$};
\node at (2.898,-1.5) {$e_2$};
\node at (0.3,0.27) {$v'$};
\node at (0.3,1.77) {$v$};

\draw (-2.598,-1.5) -- (2.598,-1.5); 
\draw (2.598,-1.5) -- (0,3); 
\draw (0,3) -- (-2.598,-1.5); 

\draw (-2.598,-1.5) -- (0,1.5); 
\draw (2.598,-1.5)  -- (0,1.5); 

\draw (-2.598,-1.5) -- (0,0); 
\draw (2.598,-1.5) -- (0,0); 
\draw (0,3) -- (0,0); 

\node at (0,-0.9) {$\tau_3$};
\node at (0.8,-0.0) {$\tau_1$};
\node at (0.8, 1.0) {$\sigma_1$};
\node at (-0.8, 1.0) {$\sigma_2$};
\node at (-0.8,-0.0) {$\tau_2$};
\end{tikzpicture}
\end{center}
Since the cones $\tau_i$ are smooth, we have a resolution 
$\rho \colon U_{\Sigma} \to U$ with two exceptional prime divisors $E$ and $E'$ corresponding to the rays generated by 
$v$ and $v'$, respectively. The  fan $\Sigma$ of the resolution consists of the cones  
$\sigma_i, \ldots, \sigma_{n-1}, \tau_1, \ldots, \tau_n$ and their faces. As above, we denote the divisors which correspond to the  rays $\mathbb R_{\geq 0}e_i$ by $D_i$.  
In compliance with  \cite[p. 381]{R87} the resolution is called the \emph{Danilov resolution}.

\begin{proposition}\label{projectivbundle} \
\begin{enumerate}
\item
The exceptional prime divisor of the toric blowup of $\frac{1}{3}(1,\ldots,1)$ is isomorphic to $\mathbb P^{n-1}$. 
\item
The exceptional prime divisor $E'$ of the Danilov resolution of $\frac{1}{3}(1,\ldots,1,2)$ is  isomorphic to $\mathbb P^{n-1}$ and the exceptional prime divisor 
 $E$ is isomorphic to the projective bundle 
 $$p_r\colon E \simeq \mathbb P\big(\mathcal O \oplus \mathcal O(2)\big) \to \mathbb P^{n-2}.$$ 
In particular,  
$$K_{E}\simeq  p_r^{\ast} \mathcal O_{ \mathbb P^{n-2}}(-n-1) \otimes \mathcal O_E\big(-2E').$$
\end{enumerate}
\end{proposition}

\begin{proof}
1) is a standard computation in toric geometry.

2) We  verify the claim about the divisor $E$ of the Danilov resolution, the analogous (but easier) computation for $E'$ we leave to the reader.
As a compact toric variety $E$ is given by the quotient lattice 
$N(v):=N/\mathbb Z v$ and the quotient cones 
$$
\overline{\tau_i}=\frac{\tau_i + \mathbb Rv}{\mathbb R v} \subset N(v)\otimes \mathbb R 
$$ and 
$$\overline{\sigma_i}=\frac{\sigma_i + \mathbb Rv}{\mathbb R v} \subset N(v)\otimes \mathbb R, \ \  1 \leq i \leq n-1,$$
 
  together with their faces. 
We denote the standard  unit vectors  of $\mathbb Z^{n-1}$ by $u_i$ and set 
$e:=u_{n-1}$ and $u_0:=-(u_1+\ldots + u_{n-2})$.

The quotient lattice $N(v)=N/\mathbb Z v$ is generated by the classes $[e_2], \ldots, [e_n]$ and 
identified with $\mathbb Z^{n-2} \times \mathbb Z$ under  the $\mathbb Z$-linear map 
$$
\phi \colon N(v) \to \mathbb Z^{n-1},
\ \  \  [e_i] \mapsto 
\begin{cases}
u_{i-1}, &  2 \leq i \leq n-1 \\
-e, &  i=n.
\end{cases}
$$
Since $e_1 = 3v-e_2- \ldots -2e_n$ and $v'=2v-e_n$, we have $\phi ([e_1])=u_0+2e$ and $\phi ([v'])=e$. 
The  $\mathbb R$-linear extension of $\phi $  identifies  $N(v)\otimes \mathbb R$ with $\mathbb R^{n-1}$, which  
allows us to view the quotient cones as cones in $\mathbb R^{n-1}$: 
\begin{align*}
	\overline{\tau_i} & \simeq \cone(u_0+2e,u_1,\ldots, \widehat{u_{i-1}}, \ldots, u_{n-2},e), \\
	\overline{\sigma_i} & \simeq \cone(u_0+2e,u_1,\ldots, \widehat{u_{i-1}}, \ldots, u_{n-2},-e). 
\end{align*}
According to \cite[Example 7.3.5]{CLS11} these cones, and their faces, build the fan of $ E \simeq \mathbb P\big(\mathcal O \oplus \mathcal O(2)\big)$. The bundle map 
$p_r \colon E \to \mathbb P^{n-2}$ is induced by the projection 
$\mathbb Z^{n-2} \times \mathbb Z \to \mathbb Z^{n-2}$ onto the first $n-2$ coordinates. 
The adjunction formula and \cite[Theorem 8.2.3]{CLS11} yield
$K_{E} \simeq\mathcal O_E(-D_1 - \ldots - D_n -E').$ 
Finally, by 
$$0 \sim_{lin} \divi(e_1-ne_2+e_3+ \ldots + e_n)= D_1-nD_2 + D_3 + \ldots + D_n -E' $$
and $p_r^{\ast} \mathcal O_{ \mathbb P^{n-2}}(1) \simeq \mathcal O_E(D_2),$
see  \cite[Proposition 6.2.7]{CLS11}, we can write the canonical bundle as 
\begin{align*}
	K_{E} \simeq & ~ \mathcal O_E(-(n+1)D_2 -2E') \\
	 \simeq & ~  p_r^{\ast} \mathcal O_{ \mathbb P^{n-2}}(-n-1) \otimes \mathcal O_E(-2E').
\end{align*}
\end{proof}

\noindent 
\begin{rem}
To illustrate the above proof observe first that for  $n=3$ we obtain the  Hirzebruch surface $\mathbb F_2$ as exceptional divisor $E$ of the Danilov resolution 
of $\frac{1}{3}(1,1,2)$, cf. \cite[Example 3.1.16]{CLS11}. The projection of the cones onto the $x$-axis induces the bundle map as can be seen in the following picture. 

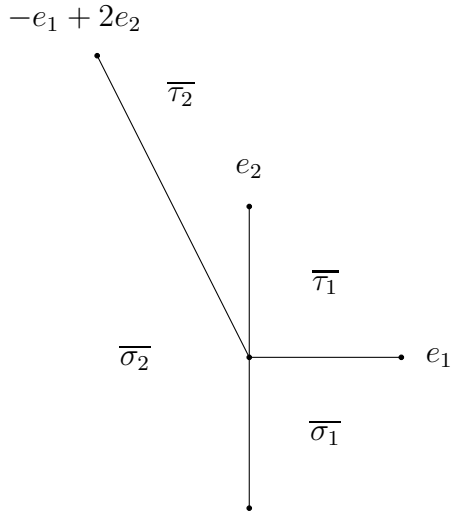
\begin{center}
\begin{tikzpicture}
\draw[fill] (0,0) circle [radius=0.03];
\draw[fill] (0,2) circle [radius=0.03];
\draw[fill] (2,0)  circle [radius=0.03];
\draw[fill] (-2,4) circle [radius=0.03];
\draw[fill] (0,-2) circle [radius=0.03];

\node at (2.5,0.0) {$e_1$};
\node at (0.0,2.5) {$e_2$};
\node at (-2.3,4.5) {$-e_1+2e_2$};

\draw (0,-2) -- (0,0); 
\draw (0,-0) -- (2,0); 
\draw (0,0) -- (0,2); 
\draw (0,0) --  (-2,4); 

\node at (1,1) {$\overline{\tau_1}$};
\node at (-1.5,0.0) {$\overline{\sigma_2}$};
\node at (1,-1) {$\overline{\sigma_1}$};
\node at (-0.9,3.5) {$\overline{\tau_2}$};

\end{tikzpicture}
\end{center}
\end{rem}

\noindent
To prove the isomorphism  $\rho_{\ast} \Theta_{U_{\Sigma}} \simeq \Theta_U$ for the toric blowup of 
 $\frac{1}{3}(1,\ldots,1)$ and for  the Danilov resolution of $\frac{1}{3}(1,\ldots,1,2)$, we consider the 
 slightly more 
general situation of a toric resolution $\rho \colon U_{\Sigma} \to U$ of an affine 
$\mathbb Q$-factorial toric variety. Recall that $U$ is $\mathbb Q$-factorial if and only if the defining cone is 
simplicial i.e. it's minimal generators are $\mathbb R$-linearly independent \cite[Proposition 4.2.7]{CLS11}. 
Our aim is to give a  combinatorial criterion for the inclusion 
$\rho_{\ast} \Theta_{U_{\Sigma}} \subset \Theta_U$ being an 
isomorphism. W.l.o.g. we may assume that $U$ has no torus factors. Then, according to \cite[Theorem 11.4.8]{CLS11}, the variety $U$ is 
an abelian quotient singularity. 
Conversely let $G \subset GL(n,\mathbb C)$ be a finite abelian group without quasi-reflections, then after simultaneous diagonalization 
each element $g \in G$ acts on $\mathbb C^n$ in the following way: 
\[
g=\diag\big(\xi^{\alpha_1(g)}, \ldots, \xi^{\alpha_n(g)}\big), 
\]
where $\xi$ is a primitive $|G|$-th root of unity and  $0 \leq \alpha_i(g) \leq |G| -1$. 
In analogy to the cyclic case, the quotient   $U=\mathbb C^n/G$ is the affine $\mathbb Q$-factorial toric variety given by 
the cone $\sigma:=\cone(e_1,\ldots, e_n)$ and the lattice 
\[
N = \mathbb Z^n + \sum_{g \in G} \frac{\mathbb Z}{|G|}\big(\alpha_1(g), \ldots, \alpha_n(g)\big) \subset \mathbb R^n.
\]

 \begin{proposition}
Let $\rho \colon U_{\Sigma} \to U$ be a toric resolution of an abelian quotient singularity. 
Let $D_i \subset U_{\Sigma}$ and $D_i' \subset U$ be the divisors corresponding to the rays 
$\mathbb R_{\geq 0}e_i$. Then, the inclusion $\rho_{\ast} \Theta_{U_{\Sigma}} \subset \Theta_U$ 
is an isomorphism if and only if the polyhedra 
$P_{D_i}$ and $P_{D_i'}$ contain the same integral points i.e 
\[
P_{D_i} \cap N^{\vee} = P_{D_i'} \cap N^{\vee} \qquad \makebox{for all} \quad  1 \leq i \leq n.  
\]
\end{proposition}

\begin{proof}

By Remark \ref{resrem} $(2)$, we shall show that $\rho_{\ast}\Theta_{U_{\Sigma}}$ is reflexive if and only if $P_{D_i} \cap N^{\vee} = P_{D_i'} \cap N^{\vee} $. 

The rays of the fan $\Sigma$ are $\mathbb R_{\geq 0} e_i$, 
together 
with $k$ rays $\mathbb R_{\geq 0} v_i$, where $v_i \in N$ is primitive. These rays correspond to $k$ exceptional prime divisors 
$E_i$ of $\rho$. 
On $U_{\Sigma}$ we have an Euler sequence (cf. \cite[Theorem 8.1.6.]{CLS11}): 
\[
0 \to \mathcal O_{U_{\Sigma}}^{\oplus k} \to \bigoplus_{i=1}^n  \mathcal O_{U_{\Sigma}}(D_i)  \oplus 
\bigoplus_{j=1}^k \mathcal O_{U_{\Sigma}}(E_j) \to \Theta_{U_{\Sigma}} \to 0.  
\]
After pushforward, the sequence 
\begin{equation}\label{ex1}
0 \to \mathcal O_{U}^{\oplus k} \to \bigoplus_{i=1}^n  \rho_{\ast}\mathcal O_{U_{\Sigma}}(D_i)  \oplus 
\bigoplus_{j=1}^k \rho_{\ast}\mathcal O_{U_{\Sigma}}(E_j) \to \rho_{\ast}\Theta_{U_{\Sigma}} \to 0 
\end{equation}
is still exact, because $U$ has rational singularities. 

{\bf Claim:}
\[
\alpha \colon \mathcal O_{U}^{\oplus k} \to \bigoplus_{j=1}^k \rho_{\ast}\mathcal O_{U_{\Sigma}}(E_j) 
\]
 is an isomorphism. 
Assuming the claim,  by the exact sequence (\ref{ex1}) we have that
$$\rho_{\ast}\Theta_{U_{\Sigma}} \simeq \bigoplus_{i=1}^n  \rho_{\ast}\mathcal O_{U_{\Sigma}}(D_i).$$ 
Hence $\rho_{\ast}\Theta_{U_{\Sigma}}$    is reflexive if and only if      $\rho_{\ast}\mathcal O_{U_{\Sigma}}(D_i)$ is reflexive for all $i$. 
Since $D_i$ is the strict transform of 
$D_i'$, there is an inclusion $\rho_{\ast}\mathcal O_{U_{\Sigma}}(D_i) \subset \mathcal O_{U}(D_i')$. This inclusion is an isomorphism if and only if  
$\rho_{\ast}\mathcal O_{U_{\Sigma}}(D_i)$ is reflexive. In summary we have that the following  are equivalent:

\begin{enumerate}
\item
$\rho_{\ast}\Theta_{U_{\Sigma}} \simeq  \Theta_U$ 
\item
$\rho_{\ast}\mathcal O_{U_{\Sigma}}(D_i) \simeq  \mathcal O_{U}(D_i')$  for all $i$. 
\end{enumerate}
Since $U$ is affine and $\rho_{\ast}\mathcal O_{U_{\Sigma}}(D_i)$ and $\mathcal O_{U}(D_i')$ are coherent sheaves they are equal if and only if 
$$H^0\big(U_{\Sigma},\mathcal O_{U_{\Sigma}}(D_i)\big) \simeq H^0\big(U,\mathcal O_{U}(D_i')\big)$$
This completes the proof, since  
\[
H^0\big(U_{\Sigma},\mathcal O_{U_{\Sigma}}(D_i)\big)=\bigoplus_{u \in P_{D_i} \cap N^{\vee}} \mathbb C \chi^{u}
\]
and
\[
H^0\big(U,\mathcal O_{U}(D_i')\big)=
\bigoplus_{u \in P_{D_i'} \cap N^{\vee}} \mathbb C \chi^{u}
\]
by \cite[Proposition 4.3.3]{CLS11}.

{\em Proof of the claim.}

We follow the construction of the Euler sequence in \cite[proof of Theorem 8.1.6]{CLS11} and 
start with the exact sequence describing the 
Picard group of $U_{\Sigma}$, see \cite[Theorem 4.1.3]{CLS11}: 
\[
0 \to N^{\vee} \to \bigoplus_{i=1}^n \mathbb Z D_i  \oplus \bigoplus_{j=1}^k \mathbb Z E_j  \to \Pic(U_{\Sigma}) \to 0. 
\]
Here, the map on the left assigns to an element $u \in N^{\vee}$ the principal divisor 
\[
\divi(u)=\sum_{i=1}^n \langle u,e_i\rangle D_i + \sum_{j=1}^k \langle u,v_j\rangle E_j. 
\]
Let $m$ be the order of $G$, then $m e_l \in N^{\vee}$ for all $e_l$ i.e. we have relations 
\[
0 \sim_{lin} \divi(m e_l)= mD_l + \sum_{j=1}^k \langle m e_l,v_j\rangle E_j
\]
in the Picard group. Since $\rk\big(\Pic(U_{\Sigma})\big)=k$, the relations imply that the projection from the second summand of the sequence 
$\gamma \colon \bigoplus_{j=1}^k \mathbb Z E_j  \to \Pic(U_{\Sigma})$ becomes an isomorphism after tensoring with 
$\mathcal O_{U_{\Sigma}}$. This map, which we also denote by $\gamma$, fits into a commutative triangle, where the vertical map is the inclusion: 
\[
\xymatrix{
\mathcal O_{U_{\Sigma}}^{\oplus k} \ar[r]^{\gamma}  & \mathcal O_{U_{\Sigma}}^{\oplus k}  \\
\bigoplus_{j=1}^k \mathcal O_{U_{\Sigma}}(-E_j)  \ar[u] \ar[ru] &   }
\]
After dualizing and pushforward the diagram reads: 
\[
\xymatrix{
\mathcal O_{U}^{\oplus k} \ar[d]  & \ar[l] \ar[ld]\mathcal O_{U}^{\oplus k}  \\
\bigoplus_{j=1}^k \rho_{\ast}\mathcal O_{U_{\Sigma}}(E_j) &   }
\]
The horizontal arrow is still an isomorphism, since $\gamma$ was.  
But now, also the vertical map is an 
isomorphism, since $\rho_{\ast}\mathcal  O_{U_{\Sigma}}(E_j) \simeq \mathcal O_U$ 
and the map is induced by inclusion. 
Therefore, the diagonal map is an isomorphism as well. By construction of the Euler sequence, the diagonal map is the map
$\alpha$ from above. This proves the claim.

\end{proof}

\begin{corollary}\label{1}
Let $\rho \colon U_{\Sigma} \to U$ be the toric blowup of $\frac{1}{3}(1,\ldots,1)$ or the Danilov resolution of 
 $\frac{1}{3}(1,\ldots,1,2)$ , then it holds  
$\rho_{\ast} \Theta_{U_{\Sigma}} \simeq \Theta_U$. 
\end{corollary}

\begin{proof}
In case of the  singularity  $\frac{1}{3}(1,\ldots,1)$ we have 
\begin{align*}
	P_{D_i'}&=\lbrace x \in \mathbb R^n ~ \big\vert ~ x_i \geq -1, ~ x_j \geq 0 ~ \makebox{for} ~ i \neq j \rbrace, \\
	P_{D_i}&= P_{D_i'} \cap \lbrace x_1 + \ldots + x_n \geq 0 \rbrace. 
\end{align*}
Let $x$ be a point in the dual lattice  $N^{\vee} = \lbrace x \in \mathbb Z^n  ~ \big\vert ~ 3  ~ \makebox{divides} ~ (x_1 + \ldots + x_n) \rbrace$, which is also contained in the polyhedron
 $P_{D_i'}$. We have to show that 
$x$ satisfies the inequality $x_1 + \ldots + x_n \geq 0$. This is clear, since $x_1 + \ldots + x_n$ is an integer divisible by $3$ and greater or equal to $-1$. \newline
In case of the  singularity  $\frac{1}{3}(1,\ldots,1,2)$ the polyhedron $P_{D_i'}$ is as above and the points in $P_{D_i}$  fulfill the additional inequalities: 
\begin{align*}
	x_1 + \ldots + x_{n-1} + 2x_n & \geq 0,  \\
	 2x_1 + \ldots + 2x_{n-1} + x_n &\geq 0. 
\end{align*}
Let  $x$ be a point in $N^{\vee} = \lbrace x \in \mathbb Z^n  ~ \big\vert ~ 3  ~ \makebox{divides} ~ (x_1 + \ldots + x_{n-1} + 2x_n) \rbrace$, then 
$ 2x_1 + \ldots + 2x_{n-1} + x_n$ is divisible by $3$, since 
\[
2(x_1 + \ldots + x_{n-1} + 2x_n)= (2x_1 + \ldots + 2x_{n-1} + x_n) + 3 x_n. 
\]
Now the proof that $P_{D_i'} \cap N^{\vee} \subset P_{D_i}$ is as above. 
\end{proof}

\newpage

 \begin{proposition}\label{2}
 Let $\rho \colon U_{\Sigma} \to U$ be 
 \begin{enumerate}
 \item the toric blowup of 
 $\frac{1}{3}(1,\ldots,1)$, or 
 \item the Danilov resolution of 
 $\frac{1}{3}(1,\ldots,1,2)$.
 \end{enumerate}
  Then it holds  $R^1 \rho_{\ast} \Theta_{U_{\Sigma}}=0$. 
 \end{proposition}
 
\begin{rem}

Corollary \ref{1} and Proposition \ref{2} conclude the proof of Proposition \ref{resproposition}.
\end{rem}

\begin{proof}
1) Let  $\rho \colon U_{\Sigma} \to U$ be  the toric blowup of $\frac{1}{3}(1,\ldots,1)$.  The  Euler sequence on $U_{\Sigma}$ reads 
\[
0 \to \mathcal O_{U_{\Sigma}} \to \bigoplus_{i=1}^n  \mathcal O_{U_{\Sigma}}(D_i)  \oplus \mathcal O_{U_{\Sigma}}(E) \to \Theta_{U_{\Sigma}} \to 0.  
\]
Since $U$ has rational singularities, we obtain an isomorphism 
\[
\bigoplus_{i=1}^n  R^1 \rho_{\ast} \mathcal O_{U_{\Sigma}}(D_i)  \oplus R^1 \rho_{\ast}\mathcal O_{U_{\Sigma}}(E) 
\simeq R^1 \rho_{\ast} \Theta_{U_\Sigma}. 
\]
We need to prove that $R^1 \rho_{\ast}\mathcal O_{U_{\Sigma}}(E) =0$ and 
$R^1 \rho_{\ast} \mathcal O_{U_{\Sigma}}(D_i) =0$  for all  $i$.  
The sequence 
\[
0 \to \mathcal O_{U_{\Sigma}} \to \mathcal O_{U_{\Sigma}}(E) \to \mathcal O_E(E) \to 0
\]
gives us the isomorphism $$R^1 \rho_{\ast} \mathcal O_{U_{\Sigma}}(E) \simeq  R^1 \rho_{\ast} \mathcal O_E(E).$$ 
Since $E \simeq \mathbb P^{n-1}$, the normal bundle $\mathcal O_E(E)$ is a multiple of $\mathcal O_{\mathbb P^{n-1}}(1)$.  Consequently, since $n\geq 3$, its
 first cohomology vanishes, which implies the vanishing of $R^1 \rho_{\ast} \mathcal O_E(E)$, too. 
To show that $R^1 \rho_{\ast} \mathcal O_{U_{\Sigma}}(D_i)=0$, we use the Cartier data of  $D_i$. 
By symmetry, we may assume $i=1$. 
The Cartier data of $D_1$ is the collection of the vectors:
\[
u(\sigma_1)=0 \quad \makebox{and} \quad u(\sigma_i)=-e_1 + e_i
\quad \makebox{for} 
\quad 2 \leq i \leq n.
\]
According to  \cite[Proposition 6.1.1]{CLS11} the sheaf $\mathcal O_{U_{\Sigma}}(D_1)$ is globally generated, since all  of the vectors 
$u(\sigma_i)$ are contained in the polyhedron associated to $D_1$: 
\[
P_{D_1}=\lbrace x \in \mathbb R^n ~ \big\vert ~ x_1 \geq -1, ~ x_2 \geq 0, ~ \ldots ~ ,  x_n \geq 0, ~ x_1+ \ldots +x_n \geq 0 \rbrace.
\]
Demazure vanishing \cite[Theorem 9.2.3]{CLS11} tells us that the higher cohomology groups of  
$\mathcal O_{U_{\Sigma}}(D_1)$ vanish. 
We conclude in particular the vanishing of $R^1 \rho_{\ast} \mathcal O_{U_{\Sigma}}(D_1)$, because $U$ is affine.  \newline

\smallskip
2) Let   $\rho \colon U_{\Sigma} \to U$ be  the Danilov resolution of $\frac{1}{3}(1,\ldots,1,2)$. By the Euler sequence $R^1\rho_{\ast} \Theta_{U_{\Sigma}}=0$ if and only if 
\begin{enumerate}
\item 
$R^1 \rho_{\ast}\mathcal O_{E}(E)=0$,
\item
$R^1 \rho_{\ast}\mathcal O_{E'}(E')=0$ and 
\item
$R^1 \rho_{\ast} \mathcal O_{U_{\Sigma}}(D_i)=0$ for all  $i$. 
\end{enumerate}
Since $E' \simeq \mathbb P^{n-1}$, the normal bundle  $\mathcal O_{E'}(E') $ is a multiple of   $\mathcal O_{\mathbb P^{n-1}}(1)$ and 
$R^1 \rho_{\ast}\mathcal O_{E'}(E')$ is zero by the same argument as above. 
The vanishing of $R^1 \rho_{\ast} \mathcal O_{U_{\Sigma}}(D_i)$ is also shown as before, using  Demazure's theorem.
For the remaining sheaf $R^1 \rho_{\ast}\mathcal O_{E'}(E')$ we proceed as follows: by Proposition \ref{projectivbundle} $E$ is the projective bundle 
\[
p_r \colon E  \simeq \mathbb P(\mathcal O \oplus \mathcal O(2)) \to \mathbb P^{n-2}.
\] 
Using the linear equivalence 
\[
0 \sim_{lin} \divi(3e_2)= 3D_2 + E + 2E',
\]
 we can rewrite the normal bundle $\mathcal O_E(E)$ in the following way: 
\[ 
\mathcal O_E(E) \simeq \mathcal O_E(-3D_2) \otimes \mathcal O_E(-2E') \simeq p_r^{\ast} \mathcal O_{ \mathbb P^{n-2}}(-3) 
\otimes \mathcal O_E(-2E').
\] 
Serre duality on $E$ and the projection formula implies:
\begin{align*}
	H^1\big(E,\mathcal O_E(E)\big)^{\vee} \simeq & ~ H^{n-2}\big(E, p_r^{\ast} \mathcal O_{ \mathbb P^{n-2}}(-n+2)\big)  \\
	 \simeq & ~ H^{n-2}\big(\mathbb P^{n-2}, \mathcal O_{ \mathbb P^{n-2}}(-n+2)\big) =0.\end{align*}
\end{proof}

{\tiny MATHEMATISCHES INSTITUT, UNIVERSIT\"AT BAYREUTH, 95440 BAYREUTH, GERMANY}

{\scriptsize\emph{E-mail address: Ingrid.Bauer@uni-bayreuth.de}} \quad {\scriptsize\emph{Christian.Gleissner@uni-bayreuth.de}}

\end{document}